\documentclass[reqno]{amsart}
\usepackage{hyperref}
\usepackage{amssymb}

\allowdisplaybreaks

\theoremstyle{plain}
\newtheorem{theorem}{Theorem}[section]

\newtheorem{corollary}[theorem]{Corollary}

\theoremstyle{definition}
\newtheorem{definition}[theorem]{Definition}

\theoremstyle{remark}
\newtheorem{remark}[theorem]{Remark}
\numberwithin{equation}{section}

\begin{document}

\title[Parabolic equation with nonlinear nonlocal boundary condition]
{Initial boundary value problem for a semilinear parabolic
equation with nonlinear nonlocal boundary condition}

\author[A. Gladkov]{Alexander Gladkov}
\address{Alexander Gladkov \\ Department of Mechanics and Mathematics
\\ Belarusian State University \\ Nezavisimosti Avenue 4, 220030
Minsk, Belarus} \email{gladkoval@mail.ru}

\author[T. Kavitova]{Tatiana Kavitova}
\address{Tatiana Kavitova \\ Department of Mathematics \\ Vitebsk State University \\ Moskovskii pr. 33, 210038
Vitebsk, Belarus} \email{kavitovatv@tut.by}

\subjclass{Primary 35K20, 35K58, 35K61}
\keywords{semilinear
parabolic equation, nonlocal boundary condition, local solution,
uniqueness}

\begin{abstract}
In this paper we consider an initial boundary value problem for a
semilinear parabolic equation with nonlinear nonlocal boundary
condition. We prove comparison principle, the existence theorem of
a local solution and  study the problem of uniqueness and
nonuniqueness.
\end{abstract}

\maketitle

\section{Introduction}

In this paper we consider the initial boundary value problem for
the following semilinear parabolic equation
\begin{equation}\label{v:u}
    u_t=\Delta u+c(x,t)u^p,\;x\in\Omega,\;t>0,
\end{equation}
with nonlinear nonlocal boundary condition
\begin{equation}\label{v:g}
\frac{\partial u(x,t)}{\partial\nu}=\int_{\Omega}{k(x,y,t)u^l(y,t)}\,dy,\;x\in\partial\Omega,\;t>0,
\end{equation}
and initial datum
\begin{equation}\label{v:n}
    u(x,0)=u_{0}(x),\; x\in\Omega,
\end{equation}
where $p>0,\,l>0$, $\Omega$ is a bounded domain in $\mathbb{R}^n$
for $n\geq1$ with smooth boundary $\partial\Omega$, $\nu$ is unit
outward normal on $\partial\Omega.$

Throughout this paper we suppose that the functions
$c(x,t),\;k(x,y,t)$ and $u_0(x)$ satisfy the following conditions:
\begin{equation*}
c(x,t)\in
C^{\alpha}_{loc}(\overline{\Omega}\times[0,+\infty)),\;0<\alpha<1,\;c(x,t)\geq0;
\end{equation*}
\begin{equation*}
k(x, y, t)\in
C(\partial\Omega\times\overline{\Omega}\times[0,+\infty)),\;k(x,y,t)\geq0;
\end{equation*}
\begin{equation*}
u_0(x)\in C^1(\overline{\Omega}),\;u_0(x)\geq0\textrm{ in
}\Omega,\;\frac{\partial u_0(x)}{\partial\nu}=\int_{\Omega}{k(x,
y,0)u_0^l(y)}\,dy\textrm{ on }\partial\Omega.
\end{equation*}
A lot of articles have been devoted to the investigation of
initial boundary value problems for parabolic equations and
systems of parabolic equations with nonlinear nonlocal Dirichlet
boundary condition (see, for example,~\cite{Deng, Gladkov_Guedda,
Gladkov_Guedda2, Gladkov_Kim, Gladkov_Kim1, Gladkov_Nikitin, Liu,
Pao, Yin} and the references therein). In particular, the initial
boundary value problem for equation~(\ref{v:u}) with nonlocal
boundary condition
\begin{equation*}
    u(x,t)=\int_{\Omega}k(x,y,t)u^l(y,t)\,dy,\;x\in\partial\Omega,\;t>0,
\end{equation*}
was considered for $c(x,t)\leq0$ and $c(x,t)\geq0$
in~\cite{Gladkov_Guedda, Gladkov_Guedda2} and~\cite{Gladkov_Kim,
Gladkov_Kim1} respectively.

We note that for $p<1$ and $l<1$ the nonlinearities in
equation~(\ref{v:u}) and boundary condition~(\ref{v:g}) are
non-Lipschitzian. The problem of uniqueness and nonuniqueness for
different parabolic nonlinear equations with non-Lipschitzian data
in bounded domain has been addressed by several authors (see, for
example, \cite{B, Cortazar, CER, Escobedo_Herrero, FW,
Gladkov_Guedda2, Gladkov_Kim1, K} and the references therein).

The aim of this paper is to study problem~(\ref{v:u})--(\ref{v:n}) for any $p>0$ and $l>0.$ We prove existence of a local solution and
establish some uniqueness and nonuniqueness results.

This paper is organized as follows. In the next section we prove
comparison principle. The existence theorem of a local solution is
established in section~\ref{v}. The problem of uniqueness and
nonuniqueness for~(\ref{v:u})--(\ref{v:n}) is investigated in
section~\ref{uniq}.
\section{ Comparison principle}

In this section the theorem of positiveness of solution and
comparison principle for~(\ref{v:u})--(\ref{v:n}) will be proved.
We begin with definitions of supersolution, subsolution and
maximal solution of~(\ref{v:u})--(\ref{v:n}).

Let be $Q_T=\Omega\times(0,T),\;S_T=\partial\Omega\times(0,T)$, $\Gamma_T=S_T\cup\overline\Omega\times\{0\}$, $T>0$.
\begin{definition}\label{v:sup}
    We say that a nonnegative function $u(x,t)\in C^{2,1}(Q_T)\cap C^{1,0}(Q_T\cup\Gamma_T)$
    is a supersolution of~(\ref{v:u})--(\ref{v:n}) in $Q_{T}$ if
        \begin{equation}\label{v:sup^u}
        u_{t}\geq\Delta u+c(x, t)u^{p},\;(x,t)\in Q_T,
        \end{equation}
        \begin{equation}\label{v:sup^g}
        \frac{\partial u(x,t)}{\partial\nu}\geq\int_{\Omega}{k(x, y, t)u^l(y, t) }\,dy,\;(x,t)\in S_T,
        \end{equation}
        \begin{equation}\label{v:sup^n}
            u(x,0)\geq u_{0}(x),\; x\in\Omega,
        \end{equation}
     and $u(x,t)\in C^{2,1}(Q_T)\cap C^{1,0}(Q_T\cup\Gamma_T)$ is a subsolution of~(\ref{v:u})--(\ref{v:n}) in $Q_{T}$ if $u\geq0$ and it satisfies~(\ref{v:sup^u})--(\ref{v:sup^n}) in the reverse order. We say that $u(x,t)$ is a solution of problem~(\ref{v:u})--(\ref{v:n}) in $Q_T$ if $u(x,t)$ is both a subsolution and a supersolution of~(\ref{v:u})--(\ref{v:n}) in $Q_{T}$.
\end{definition}
\begin{definition}\label{v:max1}
     We say that a solution $u(x,t)$ of~(\ref{v:u})--(\ref{v:n}) in $Q_{T}$ is maximal solution if for any other solution $v(x,t)$ of~(\ref{v:u})--(\ref{v:n}) in $Q_{T}$ the inequality $v(x,t)\leq u(x,t)$ is satisfied for $(x,t)\in Q_T\cup\Gamma_T$.
\end{definition}
\begin{theorem}\label{p:theorem:positive}
   Suppose that $u_0\not\equiv0$ in $\Omega$ and $u(x,t)$ is a solution of~(\ref{v:u})--(\ref{v:n}) in $Q_{T}$. Then $u(x,t)>0$ in $Q_T\cup S_T$.
\end{theorem}
\begin{proof}
Since $u_0(x)\not\equiv0$ in $\Omega$ and $u_t-\Delta
u=c(x,t)u^p\geq0$ in $Q_T$, by the strong maximum principle
$u(x,t)>0$ in $Q_T.$ Let $u(x_0,t_0)=0$ in some point
$(x_0,t_0)\in S_T.$ Then according to Theorem~3.6 of~\cite{Hu} it
yields $\partial u(x_0,t_0)/\partial\nu<0$, which
    contradicts~(\ref{v:g}).
\end{proof}

\begin{theorem}\label{p:theorem:comp-prins}
     Let $u(x,t)$ and $v(x,t)$ be a supersolution and a subsolution of problem~(\ref{v:u})--(\ref{v:n}) in $Q_{T}$, respectively, with $u(x,0)\geq v(x,0)$ in $\Omega$. Suppose that $u(x,t)>0$ or $v(x,t)>0$ in $Q_T\cup\Gamma_T$ if $\min(p,l)<1$. Then $u(x,t)\geq v(x,t)$ in $Q_T\cup\Gamma_T$.
\end{theorem}
\begin{proof}
Let $\varphi(x,\tau)\in C^{2,1}(\overline{Q}_{t})\;(0<t< T)$ be a
nonnegative function which satisfies homogeneous Neumann boundary
condition. Multiplying (\ref{v:sup^u}) by $\varphi$ and then
integrating over $Q_t,$ we
     obtain
    \begin{eqnarray}\label{p:vsp1}
       &&\int_\Omega{u(x,t)\varphi(x,t)}\,dx      \geq\int_\Omega{u(x,0)\varphi(x,0)}\,dx\\
       &&+\int_0^t\int_\Omega{\left(u(x,\tau)\varphi_\tau(x,\tau)
       +u(x,\tau)\Delta\varphi(x,\tau)+c(x,\tau)u^p(x,\tau)\varphi(x,\tau)\right)}\,dx\,d\tau\nonumber\\
       &&+\int_0^t{\int_{\partial\Omega}{\varphi(x,\tau)\int_{\Omega}{k(x,y,\tau)u^l(y,\tau)}\,dy}}\,dS_x\,d\tau.\nonumber
    \end{eqnarray}
    On the other hand, the subsolution $v(x,t)$ satisfies~(\ref{p:vsp1}) with reversed inequality
     \begin{eqnarray}\label{p:vsp2}
        &&\int_\Omega{v(x,t)\varphi(x,t)}\,dx\leq\int_\Omega{v(x,0)\varphi(x,0)}\,dx\\
        &&+\int_0^t\int_\Omega{\left(v(x,\tau)\varphi_\tau(x,\tau)+v(x,\tau)\Delta\varphi(x,\tau)+c(x,\tau)v^p(x,\tau)\varphi(x,\tau)\right)}\,dx\,d\tau\nonumber\\
        &&+\int_0^t{\int_{\partial\Omega}{\varphi(x,\tau)\int_{\Omega}{k(x,y,\tau)v^l(y,\tau)}\,dy}}\,dS_x\,d\tau.\nonumber
    \end{eqnarray}
Put $M=\max(\max\limits_{\overline Q_t}u,\max\limits_{\overline
Q_t}v)$ and $w(x,t)=v(x,t)-u(x,t)$. Subtracting (\ref{p:vsp1})
from~(\ref{p:vsp2}) and using mean value theorem, we get
    \begin{eqnarray}\label{p:vsp3}
        &&\int_\Omega{w(x,t)\varphi(x,t)}\,dx\leq\int_\Omega{w(x,0)\varphi(x,0)}\,dx\\
        &&+\int_0^t\int_\Omega{w(x,\tau)\left(\varphi_\tau(x,\tau)+\Delta\varphi(x,\tau)+p\theta_1^{p-1}(x,\tau)c(x,\tau)\varphi(x,\tau)\right)}\,dx\,d\tau\nonumber\\
        &&+l\int_0^t{\int_{\partial\Omega}{\varphi(x,\tau)\int_{\Omega}{\theta_2^{l-1}(y, \tau)k(x, y, \tau)w(y,\tau)}\,dy}}\,dS_x\,d\tau,\nonumber
    \end{eqnarray}
    where $\theta_i(x,\tau)\;(i=1,2)$ are some positive continuous functions in $\overline Q_t$ if $\min(p,l)<1$ and some nonnegative continuous functions in $\overline Q_t$ otherwise.

The function $\varphi(x,\tau)$ is defined as a solution of the
following problem
    \begin{equation*}
        \varphi_\tau+\Delta\varphi+p\theta_1^{p-1}(x,\tau)c(x,\tau)\varphi=0,\;(x,\tau)\in Q_t,
    \end{equation*}
    \begin{equation*}
        \frac{\partial\varphi(x,\tau)}{\partial\nu}=0,\;(x,\tau)\in S_t,
    \end{equation*}
    \begin{equation*}
        \varphi(x,t)=\psi(x),\;x\in\Omega,
    \end{equation*}
    where $\psi(x)\in C^\infty_0(\Omega),\;0\leq\psi\leq1$. By virtue of the comparison principle for linear parabolic equations the solution $\varphi(x,\tau)$ of this problem is nonnegative and bounded. By~(\ref{p:vsp3}) and $w(x,0)\leq0$ we have
    \begin{equation}\label{p:vsp4}
        \int_\Omega{w(x,t)\psi(x)}\,dx\leq m\int_0^t{{\int_{\Omega}{w_+(x,\tau)}\,dx}}\,d\tau,
    \end{equation}
    where $s_+=\max(0,s),\;m=l|\partial\Omega|\sup\limits_{\partial\Omega\times Q_t}k(x,y,\tau)\sup\limits_{Q_t}\theta_2^{l-1}(x,\tau)\sup\limits_{S_t}\varphi(x,\tau)$, $|\partial\Omega|$ is the Lebesgue measure of $\partial\Omega$.
    Since the inequality holds for every function $\psi(x)$, we can choose a sequence $\psi_n(x)\in C_0^\infty(\Omega)$ converging in $L^1(\Omega)$ to the function
    \begin{equation*}
    \gamma(x)=
        \begin{cases}
            1,\;w(x,t)>0,\\
            0,\;w(x,t)\leq0.
        \end{cases}
    \end{equation*}
Substituting $\psi_n(x)$ instead of $\psi(x)$ in~(\ref{p:vsp4})
and letting $n\to\infty,$ we get
    \begin{equation*}
        \int_\Omega{w_+(x,t)}\,dx\leq m\int_0^t{{\int_{\Omega}{w_+(x,\tau)}\,dx}}\,d\tau.
    \end{equation*}
    By the Gronwall inequality we obtain $w_+(x,t)\leq0$.
\end{proof}
From Theorem~\ref{p:theorem:comp-prins} the following assertion is
easily deduced.
\begin{theorem}\label{p:conseq:uniq}
    Suppose that problem~(\ref{v:u})--(\ref{v:n}) has a solution in $Q_T$ with any nonnegative initial data for $\min(p,l)\geq1$ and with positive initial data otherwise. Then the solution of~(\ref{v:u})--(\ref{v:n}) is unique in $Q_T$.
\end{theorem}

\section{ Local existence}\label{v}

In this section we establish the local existence of solution
for~(\ref{v:u})--(\ref{v:n}) using representation formula and the
contraction mapping argument.

Let $\{\varepsilon_m\}$ be decreasing to $0$ sequence such that
$0<\varepsilon_m<1$. For $\varepsilon=\varepsilon_m$ let
$u_{0\varepsilon}(x)$ be the functions with the following
properties:
\begin{gather*}
u_{0\varepsilon}(x)\in
C^1(\overline\Omega),\;u_{0\varepsilon}(x)\geq\varepsilon,\;u_{0\varepsilon_i}(x)\geq
u_{0\varepsilon_j}(x),\;\varepsilon_i\geq\varepsilon_j,\\
u_{0\varepsilon}(x)\to u_0(x)\textrm{ as
}\varepsilon\to0,\;\frac{\partial
u_{0\varepsilon}(x)}{\partial\nu}=\int_{\Omega}{k(x,y,0)u_{0\varepsilon}^l(y)}\,dy
\textrm{ for }x\in\partial\Omega.
\end{gather*}
Since the nonlinearities in~(\ref{v:u}) and (\ref{v:g}), the
Lipschitz condition is not satisfied if $\min (p,l) < 1,$ and thus
we need to consider the auxiliary problem for equation~(\ref{v:u})
with boundary condition~(\ref{v:g}) and the initial datum
\begin{equation}\label{v:nvsp}
u_\varepsilon(x,0)=u_{0\varepsilon}(x),\; x\in\Omega.
\end{equation}
\begin{theorem}
For some values of $T$ problem~(\ref{v:u}), (\ref{v:g}),
(\ref{v:nvsp}) has a unique solution in $Q_T$.
\end{theorem}
\begin{proof}
Let $G_N(x,y;t-\tau)$ be the Green function for the heat equation
with homogeneous Neumann boundary condition. We note that the
function $G_N(x,y;t-\tau)$ has the following properties (see, for
example, \cite{Kahane}):
    \begin{equation}\label{p:1G_N}
G_N(x,y;t-\tau)\geq0,\;x,y\in\Omega,\;0\leq\tau<t<T,
    \end{equation}
    \begin{equation}\label{p:3G_N}
\int_{\Omega}{G_N(x,y;t-\tau)}\,dy=1,\;x\in\Omega,\;0\leq\tau<t<T.
    \end{equation}
It is well known that problem~(\ref{v:u}), (\ref{v:g}),
(\ref{v:nvsp}) in $Q_{T}$ is equivalent to the equation
    \begin{eqnarray}\label{le:equat}
u_\varepsilon(x,t)&=&\int_\Omega{G_N(x,y;t)u_{0\varepsilon}(y)}\,dy+\int_0^t{\int_\Omega{G_N(x,y;t-\tau)c(y,\tau)u^p_\varepsilon(y,\tau)}\,dy}\,d\tau\nonumber\\
&+&\int_0^t{\int_{\partial\Omega}{G_N(x,\xi;t-\tau)\int_{\Omega}{k(\xi,y,\tau)u^l_\varepsilon(y,\tau)}\,dy}}\,dS_\xi\,d\tau\equiv
Lu_\varepsilon(x,t).
    \end{eqnarray}
To show that~(\ref{le:equat}) is solvable for small $T$ we use the
contraction mapping argument. To this end we define a sequence of
functions $\{u_{\varepsilon,n}(x,t)\},\;n=1,2,\dots$, in the
following way:
    \begin{equation}\label{p:u_1}
    u_{\varepsilon,1}(x,t)\equiv\varepsilon,\;(x,t)\in\overline Q_T,
    \end{equation}
    and
    \begin{equation}\label{le:seq}
        u_{\varepsilon,n+1}(x,t)=Lu_{\varepsilon,n}(x,t),\;(x,t)\in\overline Q_T,\;n=1,2,\dots\,.
    \end{equation}

    Set
    \begin{equation*}
        M_{0\varepsilon}=\sup\limits_{x\in\Omega}u_{0\varepsilon}(x).
    \end{equation*}
    Using the method of mathematical induction we prove that the inequalities
    \begin{equation}\label{le:induction}
    \sup\limits_{Q_{T_1}}u_{\varepsilon,n}(x,t)\leq M,\;n=1,2,\dots,
    \end{equation}
hold for some constants $T_1>0$ and
$M>\max(\varepsilon,M_{0\varepsilon})$. For $n=1$ the validity
of~(\ref{le:induction}) is obvious. Supposing
that~(\ref{le:induction}) is true for $n=m,$ we shall prove it for
$n=m+1$. Indeed, by~(\ref{p:1G_N})--(\ref{le:equat})
and~(\ref{le:seq}) we have
     \begin{eqnarray}\label{p:equation)}
u_{\varepsilon,m+1}(x,t)&=&\int_\Omega{G_N(x,y;t)u_{0\varepsilon}(y)}\,dy\nonumber\\
&+&\int_0^t{\int_\Omega{G_N(x,y;t-\tau)c(y,\tau)u_{\varepsilon,m}^p(y,\tau)}\,dy}\,d\tau\nonumber\\
&+&\int_0^t{\int_{\partial\Omega}{G_N(x,\xi;t-\tau)\int_{\Omega}{k(\xi,y,\tau)u_{\varepsilon,m}^l(y,\tau)}\,dy}\,dS_\xi}\,d\tau\nonumber\\
&\leq& M_{0\varepsilon}+M^p\nu(t)+M^l\mu(t),
    \end{eqnarray}
    where
    \begin{equation*}
        \nu(t)=\sup\limits_{x\in\Omega}\int_0^t{\int_\Omega{G_N(x,y;t-\tau)c(y,\tau)}\,dy}\,d\tau,
    \end{equation*}
    \begin{equation*}
       \mu(t)=\sup\limits_{x\in\Omega}\int_0^t{\int_{\partial\Omega}{G_N(x,\xi;t-\tau)\int_{\Omega}{k(\xi,y,\tau)}\,dy}\,}\,dS_\xi\,d\tau.
    \end{equation*}

We note that (see~\cite{Hu_Yin1}) there exist positive constants
$\delta_1$ and $a_1$ such that
    \begin{equation}\label{p:m_n1}
       \mu(t)\leq a_1\sqrt{t}\textrm{ for }t\leq\delta_1.
    \end{equation}
    Due to~(\ref{p:1G_N}), (\ref{p:3G_N}) we have
    \begin{equation}\label{p:m_n2}
      \nu(t)\leq a_2t\textrm{ for }t\leq\delta_2,
    \end{equation}
    where $\delta_2$ and $a_2$ are some positive constants.
    We choose $0<T_1<\min(\delta_1,\delta_2)$ such that
    \begin{equation}\label{p:22}
       \sup\limits_{0<t<T_1}(M^p\nu(t)+M^l\mu(t))\leq M-M_{0\varepsilon}.
    \end{equation}
    By virtue of~(\ref{p:equation)}) and (\ref{p:22}) we have~(\ref{le:induction}) with $n=m+1$.
    By~(\ref{p:1G_N})--(\ref{le:seq}) and the properties of $u_{0\varepsilon}(x)$ we get
    \begin{equation}\label{le:111}
    u_{\varepsilon,n}(x,t)\geq\varepsilon,\;(x,t)\in\overline Q_{T_1},\;n=1,2,\dots\;.
    \end{equation}
    Using mean value theorem we obtain for $n=2,3,\dots$
    \begin{eqnarray*}
&&\hspace{-0.7cm}\sup\limits_{Q_{T_1}}|u_{\varepsilon,n+1}(x,t)-u_{\varepsilon,n}(x,t)|\\
&=&\sup\limits_{Q_{T_1}}|\int_0^t{\int_\Omega{G_N(x,\xi;t-\tau)c(y,\tau)(u_{\varepsilon,n}^p(\xi,\tau)-u_{\varepsilon,n-1}^p(\xi,\tau))}\,d\xi}\,d\tau\\
&+&\int_0^t{\int_{\partial\Omega}{G_N(x,\xi;t-\tau)\int_{\Omega}{k(\xi,y,\tau)(u_{\varepsilon,n}^l(y,\tau)-u_{\varepsilon,n-1}^l(y,\tau))}\,dy}\,dS_\xi}\,d\tau       |\\
&\leq&\sup\limits_{Q_{T_1}}\left(p\theta_{1,n}^{p-1}(x,t)\nu(t)+l\theta_{2,n}^{l-1}(x,t)\mu(t)\right)\sup\limits_{Q_{T_1}}
|u_{\varepsilon,n}(x,t)-u_{\varepsilon,n-1}(x,t)|\\
&\leq&\sup\limits_{(0,T_1)}\rho(t)\sup\limits_{Q_{T_1}}|u_{\varepsilon,n}(x,t)-u_{\varepsilon,n-1}(x,t)|
\leq(M+\varepsilon)\left(\sup\limits_{(0,T_1)}\rho(t)\right)^{n-1},
    \end{eqnarray*}
where $\theta_{i,n}(x,t)\;(i=1,2)$ are continuous functions in
$\overline Q_{T_1}$ such that $\alpha_1\leq\theta_{i,n}(x,t)\leq
M_1$ for $(x,t)\in \overline Q_{T_1}$,
$\rho(t)=p(\alpha_1^{p-1}+M_1^{p-1})\nu(t)+l(\alpha_1^{l-1}+M_1^{l-1})\mu(t)$
for $t\in[0,T_1]$. We note that positive constants $\alpha_1$ and
$M_1$ do not depend on $n$. By~(\ref{p:m_n1}) and~(\ref{p:m_n2})
there exists a constant $T\in(0,T_1)$ such that
    \begin{equation*}
        \sup\limits_{(0,T)}\rho(t)<1.
    \end{equation*}
Hence, the sequence $\{u_{\varepsilon,n}(x,t)\}$ converges
uniformly in $\overline Q_T$ as $n\to\infty$. We denote
    \begin{equation*}
        u_\varepsilon(x,t)=\lim\limits_{n\to\infty} u_{\varepsilon,n}(x,t).
    \end{equation*}
By virtue of~(\ref{le:induction}), (\ref{le:111}) we have
    \begin{equation*}
       \varepsilon\leq u_\varepsilon(x,t)\leq M,\;(x,t)\in\overline Q_T.
    \end{equation*}
Passing to the limit as $n\to\infty$ in~(\ref{le:seq}) by
dominated convergence theorem we obtain that the function
$u_\varepsilon(x,t)$  satisfies~(\ref{le:equat}). Hence,
$u_\varepsilon(x,t)$ solves problem~(\ref{v:u}), (\ref{v:g}),
(\ref{v:nvsp}) in $Q_T$.

By contradiction we shall prove uniqueness of the solution
of~(\ref{v:u}), (\ref{v:g}), (\ref{v:nvsp}) in $Q_T$ for small
values of $T$. Let problem~(\ref{v:u}), (\ref{v:g}),
(\ref{v:nvsp}) have at least two solutions $u_\varepsilon(x,t)$
and $v_\varepsilon(x,t)$ in $Q_T.$ Arguing as above we get
    \begin{eqnarray*}
        &&\hspace{-0.7cm}\sup\limits_{Q_T}|u_\varepsilon(x,t)-v_\varepsilon(x,t)|\\
        &=&\sup\limits_{Q_T}|\int_0^t{\int_\Omega{G_N(x,\xi;t-\tau)c(y,\tau)(u_\varepsilon^p(\xi,\tau)-v_\varepsilon^p(\xi,\tau))}\,d\xi}\,d\tau\\
        &+&\int_0^t{\int_{\partial\Omega}{G_N(x,\xi;t-\tau)\int_{\Omega}{k(\xi,y,\tau)(u_\varepsilon^l(y,\tau)-v_\varepsilon^l(y,\tau))}\,dy}\,dS_\xi}\,d\tau|\\
        &\leq&\sup\limits_{Q_T}\left(p\theta_1^{p-1}(x,t)\nu(t)+l\theta_2^{l-1}(x,t)\mu(t)\right)\sup\limits_{Q_T}
        |u_\varepsilon(x,t)-v_\varepsilon(x,t)|\\
        &\leq&\alpha\sup\limits_{Q_T}|u_\varepsilon(x,t)-u_\varepsilon(x,t)|,
    \end{eqnarray*}
where $\theta_i(x,t)\;(i=1,2)$ are some positive continuous
functions in $\overline Q_T$ and $\alpha<1$ for small values of
$T$. Obviously, $u_\varepsilon(x,t) = v_\varepsilon(x,t)$ in
$Q_T.$
\end{proof}
\begin{theorem}\label{p:max_eq}
    For some values of  $T$ problem~(\ref{v:u})--(\ref{v:n}) has maximal solution in $Q_T$.
\end{theorem}
\begin{proof}
Let $u_\varepsilon$ be a solution of~(\ref{v:u}), (\ref{v:g}),
(\ref{v:nvsp}). It is easy to see that $u_\varepsilon$ is a
supersolution of~(\ref{v:u})--(\ref{v:n}). By
Theorem~\ref{p:theorem:comp-prins} for
$\varepsilon_1\leq\varepsilon_2$ we have $u_{\varepsilon_1}\leq
u_{\varepsilon_2}$. According to the Dini theorem
(see~\cite{Bartle_Sherbert}) for some $T>0$ the sequence
$\{u_\varepsilon(x,t)\}$ converges as $\varepsilon\to0$ uniformly
in $\overline Q_T$ to some function $u(x,t)$. Passing to the limit
as $\varepsilon\to0$ in~(\ref{le:equat}) and using dominated
convergence theorem we obtain that the function $u(x,t)$ satisfies
in $Q_T$ the following equation
\begin{eqnarray*}
        u(x,t)&=&\int_\Omega{G_N(x,y;t)u_{0}(y)}\,dy+\int_0^t{\int_\Omega{G_N(x,y;t-\tau)c(y,\tau)u^p(y,\tau)}\,dy}\,d\tau\\
        &+&\int_0^t{\int_{\partial\Omega}{G_N(x,\xi;t-\tau)\int_{\Omega}{k(\xi,y,\tau)u^l(y,\tau)}\,dy}}\,dS_\xi\,d\tau.
\end{eqnarray*}
Hence, $u(x,t)$ solves problem~(\ref{v:u})--(\ref{v:n}) in $Q_T$. It is easy to
prove that $u(x,t)$ is maximal solution
of~(\ref{v:u})--(\ref{v:n}) in $Q_T$.
\end{proof}
\section{Uniqueness and nonuniqueness}\label{uniq}

In this section we shall use some arguments of~\cite{Escobedo_Herrero} and~\cite{Gladkov_Kim1}.
\begin{theorem}\label{uniq:theorem:max-positiv}
      Let $u_0(x)\equiv 0$ and $u(x,t)$ be maximal solution of~(\ref{v:u})--(\ref{v:n}) in $Q_T$. Suppose that for some $t_0\in[0,T)$
      at least one from the following conditions is fulfilled:
       \begin{equation}\label{uniq:c}
       c(x_0,t_0)>0 \textrm{ for some } x_0\in\Omega \, \textrm{ and } \, 0<p<1
       \end{equation}
or
    \begin{equation}\label{uniq:k}
    k(x,y_0,t_0)>0\textrm{ for any }x\in\partial\Omega\textrm{ and some }y_0\in\partial\Omega \, \textrm{ and } \, 0<l<1.
    \end{equation}
Then maximal solution $u(x,t)$ of problem~(\ref{v:u})--(\ref{v:n})
is nontrivial in $Q_T$.
\end{theorem}
\begin{proof}
At first we suppose that (\ref{uniq:c}) is true. By the continuity
of the function $c(x,t)$ there exist a neighborhood
$U(x_0)\subset\Omega$ of $x_0$ in $\Omega$ and a constant $T_1<T$
such that $c(x,t)\geq c_0>0$ for $x\in U(x_0)$ and
$t\in[t_0,T_1]$. We introduce the auxiliary problem
    \begin{eqnarray}\label{uniq:2}
    \left\{ \begin{array}{ll}
    u_{t}=\Delta u+c(x,t)u^p,\;x \in U(x_0),\;t_0<t<T_1,\\
    u(x,t)=0 ,\;x\in\partial U(x_0),\;t_0<t<T_1,\label{uniq:additional_system}\\
    u(x,t_0)=0,\;x\in U(x_0).
    \end{array} \right.
    \end{eqnarray}
    We shall construct a subsolution of problem~(\ref{uniq:additional_system}).
    Put $\underline{u}(x,t)=C(t-t_0)^{\frac{1}{1-p}}w(x,t)$, where $C$ is some positive constant and $w(x,t)$ is a solution of the following problem
    \begin{eqnarray}\label{unic:vsp_w}
    \left\{\begin{array}{ll}w_{t}=\Delta w,\;x\in U(x_0),\;t_0<t<T_1,\\
    w(x,t)=0,\;x\in\partial U(x_0),\;t_0<t<T_1,\\
    w(x,t_0)=w_0(x),\;x\in U(x_0).
    \end{array}\right.
    \end{eqnarray}
Here $w_0(x)$ is nontrivial nonnegative continuous function in
$\overline {U(x_0)}$ which satisfies boundary condition. We note
that $\underline{u} (x,t)=0$ if $t=t_0$ or $x\in\partial U(x_0)$.
By the strong maximum principle we get $0<w(x,t)< M_0 = \sup_{x
\in
    U(x_0)}w_0(x)$ for $x \in U(x_0)$ and $t_0<t<T_1$.

For all $(x,t)\in U(x_0)\times(t_0,T_1)$ we have
    $$
    \underline{u}_t-\Delta \underline{u}-c(x,t)\underline{u}^p=\frac{C}{1-p}(t-t_0)^\frac{p}{1-p}w-c(x,t)C^p(t-t_0)^\frac{p}{1-p}w^p \leq
    0,
    $$
where $C\leq M_0^{-1}[c_0(1-p)]^{1/(1-p)}$.

Let $u(x,t)$ be maximal solution of~(\ref{v:u})--(\ref{v:n}) in
$Q_T$ with trivial initial datum. According to
Theorem~\ref{p:max_eq}
$u(x,t)=\lim\limits_{\varepsilon\to0}u_\varepsilon(x,t)$, where
$u_\varepsilon(x,t)$ is positive supersolution
of~(\ref{v:u})--(\ref{v:n}) in $Q_T$.  It is easy to see that
$u_\varepsilon (x,t)$ is a supersolution
of~(\ref{uniq:additional_system}). By the comparison principle for
problem~(\ref{uniq:2}) we have
$u_\varepsilon(x,t)\geq\underline{u}(x,t)$ for $(x,t)\in\overline
{U(x_0)}\times[t_0,T_1)$. Passing to the limit as
$\varepsilon\to0$ we get $u(x,t)\geq\underline{u}(x,t)$ for
$(x,t)\in\overline {U(x_0)}\times[t_0,T_1)$. By~(\ref{v:g}) and
the strong maximum principle we obtain that maximal solution
$u(x,t)>0$ for all $x\in\overline\Omega$ and $t_0<t<T_1$.

Now we suppose that~(\ref{uniq:k}) is realized. Then there exist a
neighborhood $V(y_0)\subset\overline\Omega$ of $y_0$ and a
constant $T_2\in(t_0,T)$ such that $k(x,y,t)>0$ for $t_0\leq t\leq
T_2,\;x\in\partial\Omega$ and $y\in V(y_0)$.

We use the change of variables in a neighborhood of $\partial
\Omega$ as in~\cite{CPE}. Let $\overline x$ be a point in
$\partial \Omega$ and $\widehat{n}(\overline x)$ be the unit inner
normal to $\partial \Omega$ at the point $\overline x$. Since
$\partial \Omega$ is smooth it is well known that there exists
$\delta >0$ such that the mapping
$\psi:\partial\Omega\times[0,\delta]\to \mathbb{R}^n$ given by
$\psi(\overline x,s)=\overline x +s\widehat{n}(\overline x)$
defines new coordinates $(\overline x,s)$ in a neighborhood of
$\partial\Omega$ in $\overline\Omega$. A straightforward
computation shows that, in these coordinates, $\Delta$ applied to
a function $g(\overline x,s)=g(s)$, which is independent of the
variable $\overline x$, evaluated at a point $(\overline x,s)$ is
given by
    \begin{equation}\label{uniq:new-coord}
        \Delta g(\overline x,s)=\frac{\partial^2g}{\partial s^2}(\overline x,s)-\sum_{j=1}^{n-1}\frac{H_j(\overline x)}{1-s
        H_j (\overline x)}\frac{\partial g}{\partial s}(\overline x,s),
    \end{equation}
where $H_j(\overline x)$ for $j=1,\dots,n-1,$ denote the principal
curvatures of $\partial\Omega$ at $\overline x$.

Let $\alpha>1/(1-l),\;0<\xi_0\leq1$ and
$t_0<T_3\leq\min(T_2,t_0+\delta^2)$.
    For points in $Q_{\delta,T_3}=\partial\Omega\times[0,\delta]\times(t_0,T_3)$ of coordinates $(\overline x,s,t)$ we define
    \begin{equation*}
        \underline{u}(\overline x,s,t)=(t-t_0)^\alpha\left(\xi_0-\frac{s}{\sqrt{t-t_0}}\right)_+^3,
    \end{equation*}
and for points in $\overline\Omega\times[t_0,T_3)\setminus
Q_{\delta,T_3}$ we put $\underline{u}(\overline x,s,t)\equiv0$. We
shall prove that $\underline{u}(\overline x,s,t)$ is subsolution
of~(\ref{v:u})--(\ref{v:n}) in $\Omega\times(t_0,T_3)$. Indeed,
using~(\ref{uniq:new-coord}), we get
    \begin{eqnarray*}
&&\hspace{-0.6cm}\underline{u}_t (\overline x,s,t)-\Delta
\underline{u}(\overline x,s,t)-c(x,t)\underline{u}^p(\overline
x,s,t)=\alpha(t-t_0)^{\alpha-1}\left(\xi_0-\frac{s}{\sqrt{t-t_0}}\right)^3_+\\
&+&\frac{3}{2}s(t-t_0)^{\alpha-3/2}\left(\xi_0-\frac{s}{\sqrt{t-t_0}}\right)^2_+-6(t-t_0)^{\alpha-1} \left(\xi_0-\frac{s}{\sqrt{t-t_0}}\right)_+\\
&-&3(t-t_0)^{\alpha-1/2}\left(\xi_0-\frac{s}{\sqrt{t-t_0}}\right)^2_+\sum_{j=1}^{n-1}
\frac{H_j(\overline x)}{(1-s H_j (\overline
x))}-c(x,t)\underline{u}^p(\overline x,s,t)\leq 0
    \end{eqnarray*}
for sufficiently small values of $\xi_0$.

It is obvious,
    \begin{equation*}
        \frac{\partial\underline u}{\partial\nu}(\overline x,0,t)=-\frac{\partial\underline u}{\partial s}(\overline x,0,t)=3(t-t_0)^{\alpha-\frac{1}{2}}\xi_0^2.
    \end{equation*}
    For sufficiently small values of $t-t_0$ we get
    \begin{eqnarray*}
    \begin{array}{ll}
        \frac{\partial\underline u}{\partial\nu}(x,t)-\int_{\Omega}k(x,y,t)\underline u^l(y,t)\,dy = 3(t-t_0)^{\alpha-\frac{1}{2}}\xi_0^2\\
        -(t-t_0)^{\alpha l}\int_{\partial\Omega\times[0,\delta]}k(x,(\overline{y},s),t)|J(\overline y,s)|\left(\xi_0-\frac{s}{\sqrt{t-t_0}}\right)^{3l}_+
        \,d\overline y\,ds \leq 3(t-t_0)^{\alpha-\frac{1}{2}}\xi_0^2\\
        -(t-t_0)^{\alpha l+\frac{1}{2}} \int_{\partial\Omega}\,d\overline y\,\int_0^{\xi_0} k(x,(\overline{y},z\sqrt{t-t_0}),t)
        |J(\overline y,z\sqrt{t-t_0})|\left(\xi_0-z \right)^{3l}_+ dz\\
        \leq 3(t-t_0)^{\alpha-\frac{1}{2}}\xi_0^2-C(t-t_0)^{\alpha l+\frac{1}{2}}\leq0,
     \end{array}
    \end{eqnarray*}
    where $J(\overline y,s)$ is Jacobian of the change of variables, and the constant $C$ does not depend on $t$. Completion of the proof is the same as in the first part of the theorem.
\end{proof}
  Suppose that
    \begin{equation}\label{unic:c2}
        c(x,t)\not\equiv 0 \textrm{ in } Q_\tau \textrm{ for any } \tau >0 \, \textrm{ and } \, 0<p<1
    \end{equation}
    and there exist sequences $\{t_k\}$ and $\{y_k\}$, $k\in N$, such that
    \begin{equation}\label{unic:c222}
    t_k>0, \lim_{k \to \infty} t_k = 0,\;y_k\in\partial\Omega,\;
    k(x,y_k,t_k)>0\textrm{ for any }x \in\partial\Omega \, \textrm{ and } \,
    0<l<1.
    \end{equation}
\begin{remark}\label{posit}
    Let the assumptions of Theorem~\ref{uniq:theorem:max-positiv} hold but only at least one condition (\ref{unic:c2}) or (\ref{unic:c222}) is fulfilled instead of (\ref{uniq:c}), (\ref{uniq:k}).
    Then maximal solution of~(\ref{v:u})--(\ref{v:n}) is positive in $Q_T\cup S_T$.
\end{remark}
\begin{corollary}\label{uniq:conseq:uniq}
   Let the assumptions of Theorem~\ref{uniq:theorem:max-positiv}  hold but only at least one condition (\ref{unic:c2}) or (\ref{unic:c222}) is fulfilled instead of (\ref{uniq:c}), (\ref{uniq:k}). Suppose that there exists $\overline t\in(0,T)$  such that
\begin{equation}\label{uniq:decrease}
   c(x,t)\textrm{ and }k(x,y,t)\textrm{ are nondecreasing with respect to } t \in[0,\overline t].
\end{equation}
Then there exists exactly one solution of~(\ref{v:u})--(\ref{v:n})
which is positive in $Q_T\cup S_T.$
\end{corollary}
\begin{proof}
Denote $u(x,t)$ maximal solution of~(\ref{v:u})--(\ref{v:n})
with $u_0(x)\equiv0$. Due to Remark~\ref{posit} $u(x,t) >0$ in
$Q_T\cup S_T.$ Suppose, for a contradiction, that there exists
another solution $v(x,t)$ of~(\ref{v:u})--(\ref{v:n}) with trivial
initial datum which is positive in $Q_T\cup S_T.$ By virtue
of~(\ref{uniq:decrease}) $v(x,t+\tau)$ is positive supersolution
of~(\ref{v:u})--(\ref{v:n}) in $Q_{\overline t-\tau}$ for any
$\tau\in(0,\overline t)$. By Theorem~\ref{p:theorem:comp-prins}
then we get $u(x,t)\leq v(x,t+\tau)$ for $(x,t)\in Q_{\overline
t-\tau}\cup\Gamma_{\overline t-\tau}$. Passing to the limit as
$\tau\to0$ we have $u(x,t)\leq v(x,t)$ for $(x,t)\in Q_{\overline
t}\cup\Gamma_{\overline t}$. By Definition~\ref{v:max1} and
Theorem~\ref{p:conseq:uniq} we obtain $v(x,t)=u(x,t)$ for all
$t\in(0,T)$.
\end{proof}
\begin{theorem}\label{uniq not triv}
Let $\min(p,l)<1,\,$ $u_0\not\equiv 0$  and (\ref{uniq:decrease})
be satisfied.
Suppose that~(\ref{unic:c2}) or~(\ref{unic:c222}) hold.
   Then the solution of~(\ref{v:u})--(\ref{v:n}) is unique.
\end{theorem}
\begin{proof}

In order to prove uniqueness we show that if  $v$ is any solution
of~(\ref{v:u})--(\ref{v:n}) then
\begin{equation}\label{uniq:1}
    u\leq v\textrm{ in }Q_{T_1},
\end{equation}
where $u$ is maximal solution of~(\ref{v:u})--(\ref{v:n}).

We shall consider three cases: $0<l<1$ and $0<p\leq 1$, $0<l<1$
and $p>1$, $0<p<1$ and $l \geq 1.$

Let $0<l<1,\;0<p\leq 1$.
Put $z=u-v$. Then $z$ satisfies the problem
\begin{equation}\label{uniq:z}
    \begin{cases}
        z_{t}\leq\Delta z + c(x,t)z^p,\;(x,t)\in Q_{T_1},\\
        \frac{\partial z(x,t)}{\partial\nu}\leq\int_{\Omega}k(x,y,t)z^l(y,t)\,dy,\;(x,t)\in S_{T_1},\\
        z(x,0) \equiv 0,\;x\in\Omega.
    \end{cases}
\end{equation}
By Corollary~\ref{uniq:conseq:uniq} there exists unique solution $h(x,t)$ of the following problem
\begin{equation*}
    \begin{cases}
        h_{t}=\Delta h + c(x,t)h^p,\;(x,t)\in Q_{T_2},\\
        \frac{\partial h(x,t)}{\partial\nu}=\int_{\Omega}k(x,y,t)h^l(y,t)\,dy,\;(x,t)\in S_{T_2},\\
        h(x,0) \equiv 0,\;x\in\Omega,
    \end{cases}
\end{equation*}
such that $h(x,t)>0$ for $x \in \overline{\Omega}$ and $0<t<T_2$.
Let $T_3=\min(T_1,T_2)$. In a similar way as in
Corollary~\ref{uniq:conseq:uniq} and
Theorem~\ref{p:theorem:comp-prins} it can be shown that $h\geq z$
and $u\geq h$. Put $a=h-z$ and use the following inequality (see,
for example,~\cite{Aguirre_Escobedo})
\begin{equation*}
h^q-u^q+v^q \geq (h-u+v)^q,
\end{equation*}
where $0<q\leq1$ and $\max\{h,v\}\leq u\leq h+v$. Then we get
\begin{equation*}
    \begin{cases}
        a_{t}\geq\Delta a + c(x,t)a^p,\;(x,t)\in Q_{T_3},\\
        \frac{\partial a(x,t)}{\partial\nu}\geq\int_{\Omega}k(x,y,t)a^l(y,t)\,dy,\;(x,t)\in S_{T_3},\\
        a(x,0) \equiv 0,\;x\in\Omega.
    \end{cases}
\end{equation*}
We claim that $a>0$ in $Q_{T_3}$. Indeed, otherwise by
Theorem~\ref{p:theorem:positive} there exists
$\bar{t}\in(0,{T_3})$ such that $a(x,t)\equiv 0$ for $(x,t)\in
Q_{\bar{t}}$.  Then we obtain
\begin{eqnarray*}
&&\hspace{-0.7cm}\int_{\Omega}k(x,y,t)(h^l(y,t)+v^l(y,t))\,dy=\frac{\partial h(x,t)}{\partial\nu}+\frac{\partial v(x,t)}{\partial\nu}=\frac{\partial z(x,t)}{\partial\nu}+\frac{\partial v(x,t)}{\partial\nu}\\
&=&\frac{\partial u(x,t)}{\partial\nu}=\int_{\Omega}k(x,y,t)u^l (y,t)\,dy=\int_{\Omega}k(x,y,t)(z(y,t)+v(y,t))^l\,dy\\
&=&\int_{\Omega}k(x,y,t)(h(y,t)+v(y,t))^l\,dy
\end{eqnarray*}
for all $x\in\partial\Omega$ and $0<t<\bar{t}$. This is a
contradiction, if (\ref{unic:c222}) is  satisfied, since
$h>0,\;v>0$ in $Q_{\bar{t}}$, $k(x,y_k,t_k)>0$ for all
$x\in\partial\Omega$ and some $y_k\in\partial\Omega$,
$0<t_k<\bar{t}$ and $0<l<1$.

If (\ref{unic:c2}) is fulfilled, we can get a contradiction in
another way. Really,
\begin{eqnarray*}
&&\hspace{-0.7cm}c(x,t)(h+v)^p=c(x,t)(z+v)^p=c(x,t)u^p=u_{t}-\Delta u=(z+v)_t-\Delta(z+v)\\
&=&(h+v)_t-\Delta(h+v)=c(x,t)(h^p+v^p)
\end{eqnarray*}
for all $x\in\Omega$ and $0<t<\bar{t}$. This is a contradiction
since $h>0,\;v>0$ in $Q_{\bar{t}}$, $c(x_1,t_1)>0$ for some $x_1
\in \Omega, \,$ $t_1 \in (0,\bar{t}),$ and $0<p<1.$

Since $a>0$ in $Q_{\bar{t}}$ by Corollary~\ref{uniq:conseq:uniq}
and Theorem~\ref{p:theorem:comp-prins} we conclude that $a\geq h$
in $Q_{\bar{t}}\cup\Gamma_{\bar{t}}.$ This implies~(\ref{uniq:1})
for the case $0<l<1$ and $0<p\leq 1$.

We consider the second case $0<l<1$ and $p>1$. It is easy to see that there exists a constant $M>0$ such that
\begin{equation*}
u^p(x,t)-v^p(x,t)\leq M(u(x,t)-v(x,t)),\;(x,t)\in Q_{T_4},
\end{equation*}
where $T_4<T_2$. Put $z=u-v$. Then the function $z(x,t)$ satisfies
problem~(\ref{uniq:z}) with $p=1$, and  $Mc(x,t)$ instead of
$c(x,t)$. Further the proof is the same as in the first case with
$p = 1$.

In the third case $0<p<1$ and $l \geq 1$  either $0<p<1$ and
$c(x,t) \equiv 0$ in $Q_\sigma$ for some $\sigma >0$ or
(\ref{unic:c2}) is true. If (\ref{unic:c2}) is fulfilled, we can
argue as in previous cases, otherwise, the solution
of~(\ref{v:u})--(\ref{v:n}) is unique by
Theorem~\ref{p:theorem:positive} and Theorem~\ref{p:conseq:uniq}.
\end{proof}
\begin{remark}
As we can see from the proof of Theorem~\ref{uniq not triv}, the
solution of~(\ref{v:u})--(\ref{v:n}) is unique  if $u_0\not\equiv
0,\,$  (\ref{uniq:decrease}) hold and $k(x,y,t) \equiv 0$ in
$\partial\Omega \times Q_{\tau}$ for some $\tau >0.$
\end{remark}


\begin{thebibliography}{99}

        \bibitem{Aguirre_Escobedo} J.\,Aguirre and M.\,Escobedo,
        {\it A Cauchy problem for $u_t-\Delta u=u^p:$ Asymptotic behavior of solutions},
        Ann. Fac. Sci. Toulouse {\bf 8} (1986-87), 175--203.

        \bibitem{Bartle_Sherbert} R.\,Bartle and D.\,Sherbert,
        {\it Introduction to real analysis},
        John Wiley $\&$ Sons, Inc. (2011), 402~pp.

        \bibitem{B} P.\,Bokes,{\it A uniqueness result for a semilinear parabolic system},
        J. Math. Anal. Appl. {\bf 331} (2007), 567--584.

        \bibitem{CPE} C.\,Cortazar,\;M.\, del Pino and M.\,Elgueta,
        {\it On the short-time behaviour of the free boundary of a porous medium equation},
        Duke Math. J. {\bf 87} (1997), no.~1, 133--149.

        \bibitem{Cortazar} C.\,Cortazar,\;M.\,Elgueta and J.~D.\, Rossi,
        {\it Uniqueness and non-uniqueness for a system of heat equations with non-linear coupling at the boundary},
        Nonlinear Anal. {\bf 37} (1999), no.~2,  257--267.

        \bibitem{CER} C.\, Cortazar,\; M.\, Elgueta and J. D.\, Rossi,
        {\it Uniqueness and nonuniqueness for the porous medium equation with non linear boundary condition},
        Diff. Int. Equat. {\bf 16} (2003), no.~10, 1215--1222.

        \bibitem{Deng} K.\,Deng,
        {\it Comparison principle for some nonlocal problems},
        Quart. Appl. Math. {\bf 50} (1992), no.~3, 517--522.

        \bibitem{Escobedo_Herrero} M.\,Escobedo and M.\,A.\,Herrero,
        {\it A semilinear parabolic system in a bounded domain},
        Ann. Mat. Pura Appl. {\bf CLXV} (1993), 315--336.

        \bibitem{FW} H.\,Fujita and S.\,Watanabe,
        {\it On the uniqueness and non-uniqueness of solutions of initial value problems for some quasi-linear parabolic equations},
        Comm. Pure Appl. Math. {\bf 21} (1968), no.~6, 631--652.

        \bibitem{Hu} B.\,Hu,
        {\it Blow-up theories for semilinear parabolic equations},
        Lecture Notes in Mathematics {\bf 2018} (2011), 127~p.

        \bibitem{Hu_Yin1} B.\,Hu and H.-M.\,Yin,
        {\it Critical exponents for a system of heat equations coupled in a non-linear boundary condition},
        Math. Methods Appl. Sci. {\bf 19} (1996), no.~14, 1099--1120.

        \bibitem{Gladkov_Guedda} A.\,Gladkov and M.\,Guedda,
        {\it Blow-up problem for semilinear heat equation with absorption and a nonlocal boundary condition},
        Nonlinear Anal. {\bf 74} (2011), no.~13, 4573--4580.

        \bibitem{Gladkov_Guedda2} A.\,Gladkov and M.\,Guedda,
        {\it Semilinear heat equation with absorption and a nonlocal boundary condition},
         Appl. Anal. {\bf 91} (2012), no.~12, 2267--2276.

        \bibitem{Gladkov_Kim} A.\,Gladkov and K.\,Ik\,Kim,
        {\it Blow-up of solutions for semilinear heat equation with nonlinear nonlocal boundary condition},
        J. Math. Anal. Appl. {\bf 338} (2008), 264--273.

        \bibitem{Gladkov_Kim1} A.\,Gladkov and K.\,Ik\,Kim,
        {\it Uniqueness and nonuniqueness for reaction-diffusion equation with nonlocal boundary condition},
        Adv. Math. Sci. Appl. {\bf 19} (2009), no.~1, 39--49.

        \bibitem{Gladkov_Nikitin} A.\,Gladkov and A.\,Nikitin,
        {\it A reaction-diffusion system with nonlinear nonlocal boundary conditions},
        Int. J. Partial Differential Equations {\bf 2014} (2014), 1--10.

        \bibitem{Kahane} C.~S.\,Kahane,
        {\it On the asymptotic behavior of solutions of parabolic equations},
        Czechoslovak Math. J. {\bf 33} (1983), no.~108, 262--285.

        \bibitem{K} M. Kordo$\breve s$, {\it On uniqueness for a semilinear parabolic system coupled in
        an equation and a boundary condition}, J. Math. Anal. Appl. {\bf 298} (2004), 655--666.

        \bibitem{Liu} D.\,Liu and C.\,Mu,
        {\it Blowup properties for a semilinear reaction-diffusion system with nonlinear nonlocal boundary conditions},
        Abs. Appl. Anal. {\bf 2010} (2010), 1--17.

        \bibitem{Pao} C.~V.\,~Pao,
        {\it Asimptotic behavior of solutions of solutions of reaction-diffusion equations with nonlocal boundary conditions},
        J. Comput. Appl. Math. {\bf 88} (1998), 225--238.

        \bibitem{Yin} Y.\,~Yin,
        {\it On nonlinear parabolic equations with nonlocal boundary condition},
        J. Math. Anal. Appl. {\bf 185} (1994), no.~1, 161--174.

\end{thebibliography}
\end{document}